\documentclass[11pt]{article}
\usepackage{amsmath}
\usepackage{color}
\newcommand{\Limsup}{\mathop{{\rm Lim}\,{\rm sup}}}
\def\disp{\displaystyle}

\def\tto{\;{\lower 1pt \hbox{$\rightarrow$}}\kern -10pt
\hbox{\raise 2pt \hbox{$\rightarrow$}}\;}
\def\Hat{\widehat}
\def\hat{\widehat}

\def\Bar{\overline}
\def\ra{\rangle}
\def\la{\langle}
\def\ve{\varepsilon}
\def\B{I\!\!B}
\def\IN{I\!\!N}

\def\h{\hfill\Box}
\def\R{I\!\!R}

\def\ox{\bar{x}}

\def\oy{\bar{y}}

\def\ou{\bar{u}}

\def\inte{\mbox{\rm int}\,}

\def\gph{\mbox{\rm gph}\,}

\def\epi{\mbox{\rm epi}\,}

\def\dom{\mbox{\rm dom}\,}

\def\h{\hfill\triangle}
\def\dn{\downarrow}
\def\O{\Omega}

\def\ph{\varphi}
\def\emp{\emptyset}
\def\st{\stackrel}
\def\oR{\Bar{\R}}

\def\lm{\lambda}
\def\gg{\gamma}

\def\al{\alpha}
\def\kk{\kappa}

\def\vt{\vartheta}
\newcounter{lk}
\def\Limsup{\mathop{{\rm Lim}\,{\rm sup}}}

\begin{document}
\newtheorem{Theorem}{Theorem}[section]
\newtheorem{Proposition}[Theorem]{Proposition}
\newtheorem{Remark}[Theorem]{Remark}
\newtheorem{Lemma}[Theorem]{Lemma}
\newtheorem{Corollary}[Theorem]{Corollary}
\newtheorem{Conjecture}[Theorem]{Conjecture}
\newtheorem{Definition}[Theorem]{Definition}
\newtheorem{Example}[Theorem]{Example}
\renewcommand{\theequation}{\thesection.\arabic{equation}}
\normalsize
\def\proof{
\normalfont
\medskip
{\noindent\itshape Proof.\hspace*{6pt}\ignorespaces}}
\def\endproof{$\h$ \vspace*{0.1in}}
\begin{center}
\vspace*{0.3in} {\bf SECOND-ORDER GROWTH, TILT STABILITY,\\ AND  METRIC REGULARITY OF THE SUBDIFFERENTIAL}\\[2ex]
D. DRUSVYATSKIY\footnote{School of Operations Research and Information Engineering, Cornell University, Ithaca, NY 14853, USA, Email: dd379@cornell.edu.}, B. S. MORDUKHOVICH\footnote{Department of Mathematics, Wayne State University, Detroit, MI 48202, USA, Email: boris@math.wayne.edu. Research of this author was partially supported by the USA National Science Foundation under grant DMS-1007132, by the Australian Research Council under grant DP-12092508, and by the Portuguese Foundation of Science and Technologies under grant MAT/11109.} and T. T. A. NGHIA\footnote{Department of Mathematics, Wayne State University, Detroit, MI 48202, USA, Email: nghia@math.wayne.edu. Research of this author was partially supported by the USA National Science Foundation under grant DMS-1007132.}
\end{center}

\small{\bf Abstract.} This paper sheds new light on several interrelated topics of second-order variational analysis, both in finite and infinite-dimensional settings. We establish new relationships between second-order growth conditions on functions, the basic properties of metric regularity and subregularity of the limiting subdifferential, tilt-stability of local minimizers, and positive-definiteness/semidefiniteness properties of the second-order subdifferential (or generalized Hessian).\vspace*{0.05in}

\vspace*{0.05in}

\normalsize

\section{Introduction}
\setcounter{equation}{0}

This work is devoted to second-order variational analysis of extended-real-valued lower semicontinuous (l.s.c.) functions $f\colon X\to\oR:=(-\infty,\infty]$, which naturally appear, in particular, in problems of constrained optimization, even when the initial data is smooth. For more details we refer the reader to the books \cite{BS,M1,RW} and the recent publications mentioned below with the bibliographies therein. Despite serious achievements attained in this area, there are more questions than answers at the current stage of development in both second-order variational theory and its applications.

The main attention of this paper is paid to the study of the following three important interrelated issues in second-order variational analysis: {\em quadratic growth} of an extended-real-valued function, {\em metric regularity} properties of its limiting subdifferential, and {\em tilt stability} of its local minimizers. Although the general framework of this paper is infinite-dimensional, the major results obtained below are new even in finite dimensions.

The underlying notion of {\em tilt-stable local minimizers} was introduced by Poliquin and Rockafellar \cite{PR2} who characterized it for prox-regular and subdifferentially continuous functions on $\R^n$ via the second-order subdifferential/generalized Hessian in the sense of Mordukhovich \cite{M0}; see Section~4 for precise definitions and more details. Quite recently, Mordukhovich and Nghia \cite{MN} have extended this characterization to the settings of Hilbert spaces by using a new notion of the combined second-order subdifferential; see Section~4.

From different prospectives, Bonnans and Shapiro \cite{BS} studied a {\em second-order/quadratic growth condition} (labeled by them as ``uniform quadratic growth with respect to the tilt parametrization") for conic programs with ${\cal C}^2$-smooth data and characterized it in terms of tilt-stable minimizers of such programs in general Banach spaces. This line of development was continued by Lewis and Zhang \cite{LZ} and by Drusvyatskiy and Lewis \cite{DL} who established the equivalence of uniform quadratic growth to tilt stability of local minimizers for, respectively, ${\cal C}^2$-partly smooth functions under a certain nondegeneracy condition and the general class of prox-regular and subdifferentially continuous functions on finite dimensional spaces. The latter result was further extended with a different proof by Mordukhovich and Nghia \cite{MN} to the class of prox-regular and subdifferentially continuous functions defined on Asplund (in particular, reflexive Banach) spaces.

In other lines of development, Arag\'on and Geoffroy \cite{AG} established remarkable equivalences between appropriate uniform quadratic growth conditions for an l.s.c.\ convex function $f\colon X\to\oR$ on a Hilbert space $X$ and various {\em metric regularity} properties (see Definition~\ref{mr-def}) of the subgradient mapping $\partial f$ in the sense of convex analysis. In the case of $X=\R^n$ their characterization of strong metric regularity of $\partial f$ has been recently extended by Drusvyatskiy and Lewis \cite{DL} to the limiting subdifferential $\partial f$ (see Definition~\ref{subdif}) of any l.s.c.\ subdifferentially continuous  function; see also Lewis and Zhang \cite{LZ} for the previous result concerning prox-regular ${\cal C}^2$-partly smooth functions. Furthermore, Mordukhovich and Nghia \cite{MN} justified the equivalence between uniform quadratic growth of an l.s.c.\ function $f\colon X\to\oR$ and strong metric regularity of its limiting subdifferential, with precise modulus relationships in the growth and strong  metric regularity properties and without imposing the subdifferential continuity of $f$, in the case of an Asplund space $X$. Finally, in the very fresh preprint \cite{AG_new} Arag\'on and Geoffroy extended their characterizations from \cite{AG} to convex functions defined on arbitrary Banach spaces with some relationships between constants of the corresponding growth conditions and regularity properties.\vspace*{0.05in}

The current work aims to achieve new results in the aforementioned directions by pushing the known characterizations beyond the convex and finite-dimensional settings and all with new illuminating arguments. After presenting in Section~2 the needed preliminary material, we devote Section~3 to the comprehensive study of {\em quantitative} (with moduli) relationships between second-order growth conditions for an
extended-real-valued l.s.c.\ function on an Asplund space and metric regularity and subregularity properties of its (limiting) subdifferential listed in Definition~\ref{mr-def}. Our results for the subdifferential metric subregularity and its strong counterpart complement those concurrently and independently obtained by Arag\'on and Geoffroy \cite{AG_new} for l.s.c.\ convex functions on Banach spaces, while the proofs of some implications in \cite{AG_new} remain valid without convexity for general l.s.c.\ functions on Asplund spaces. Note to this end that our modulus estimates are essentially better than those derived in \cite{AG_new} and, in fact, solve affirmatively the conjecture formulated therein; see more discussions in Section~3.

Section~4 is largely motivated by our intention to extend the known equivalences between metric regularity and strong metric regularity properties of the gradient mappings for ${\cal C}^2$-smooth functions, characterized via the classical Hessian, to the general class of prox-regular and subdifferentially continuous functions on $\R^n$. Developing the technique initiated in \cite{MN}, we obtain such relationships (some of them are given in the Hilbert space setting) in terms of the {\em generalized Hessian} and show that they characterize tilt-stable local minimizers of such functions. However, it is done under an additional assumption on {\em positive-semidefiniteness} of the generalized Hessian that is automatic for ${\cal C}^2$-functions at local minimizers. We conjecture that this condition is necessary. A direct consequence of our approach is a much simplified proof of the aforementioned result from Poliquin and Rockafellar \cite{PR2} characterizing tilt-stable local minimizers of prox-regular functions via the generalized Hessian.

The final Section~5 contains concluding remarks concerning future research.

\section{Preliminaries}
\setcounter{equation}{0}

In this section we summarize some of the fundamental tools of variational analysis and generalized differentiation using basically standard terminology and notation; see, e.g., \cite{M1,RW}.

Throughout the paper we let $(X,\|\cdot\|)$ and $(Y,\|\cdot\|)$ denote {\em Asplund spaces}, i.e., such Banach spaces where every separable subspace has a separable dual. This class of spaces is rather broad including, e.g., every reflexive Banach spaces. The symbol $\to$ always denotes the convergence relative to the distance $d(\cdot,\cdot)$ induced by the norm while the arrow $\xrightarrow{w^*}$ signifies the weak$^{*}$ convergence in the dual space $X^*$. The closed ball centered at $x\in X$ of radius $r$ is denoted by $\B_{r}(x)$ while the closed unit and dual unit balls are denoted by $\B$ and $\B^{*}$, respectively. The {\em distance function} associated with a nonempty set $\O\subset X$ is defined by
$$
d(x;\O):=\inf_{y\in\O}d(x,y),\quad x\in\O.
$$
A set-valued mapping $F\colon X\tto Y$ between Asplund spaces, is a mapping assigning to each point $x\in X$ a subset $F(x)\subset Y$. For such a mapping the {\em domain} and {\em graph} are defined by
$$
\dom F:=\big\{x\in X\big|\;F(x)\ne\emp\big\}\;\quad\mbox{ and }\quad\;\gph F:=\big\{(x,y)\in X\times Y\big|\;y\in F(x)\big\},
$$
respectively. For a mapping $F\colon X\tto X^{*}$ between an Asplund space $X$ and its dual $X^{*}$ we define the sequential {\em Painlev\'{e}-Kuratowski outer limit} by
\begin{eqnarray}\label{pk}
\Limsup_{x{\to}\ox}F(\bar{x}):=\big\{x^{*}\in X^{*}\big|\,\exists\textrm{ seqs. }x_k\to\bar{x}\textrm{ and }x^{*}_k\xrightarrow{w^*}x^{*} \textrm{ with }x^{*}_k\in F(x_k)\textrm{ for all }\;k\big\}.
\end{eqnarray}
Since much of our analysis revolves around local properties of graphs of set-valued mappings, the following notion arises naturally.

\begin{Definition} {\bf (localizations of set-valued mappings).} {\rm A set-valued mapping $\Hat F\colon X\tto Y$ is a {\sc localization} of a mapping $F\colon X\tto Y$ around $(\ox,\oy)\in\gph F$ if there are neighborhoods $U$ of $\ox$ and $V$ of $\oy$ such that the relationship
$$
\gph\Hat F=\gph F\cap(U\times V)\quad\textrm{ holds}.
$$}
\end{Definition}

Stability theory, from a variational-analytic point of view, revolves around the basic notion of metric regularity and its close relatives. These are the key concepts we study in the current work. For more details and further references see, e.g., \cite{DR,met_ioffe,M1}.

\begin{Definition}{\bf (metric regularity and subregularity properties of mappings).}\label{mr-def} {\rm Consider a set-valued mapping $F\colon X\tto Y$ and a pair $(\ox,\oy)$ from its graph.
\begin{itemize}
\item We say that $F$ is {\sc metrically regular} with modulus $\kk>0$ around $(\ox,\oy)$ if there are neighborhoods $U$ of $\ox$ and $V$ of
$\oy$ such that the inequality
\begin{equation}\label{mr}
d\big(x;F^{-1}(y)\big)\le\kk\,d\big(y;F(x)\big)\quad\quad\mbox{ holds for all }\;x\in U\;\mbox{ and }\;y\in V.
\end{equation}
If in addition the inverse mapping $F^{-1}$ has a single-valued localization around $(\oy,\ox)$, then $F$ is {\sc strongly metrically regular} with modulus $\kk>0$ around $(\ox,\oy)$.
\item We say that $F$ is {\sc metrically subregular} with modulus $\kk>0$ at $(\ox,\oy)$ if there is a neighborhood $U$ of $\ox$ such that
the inequality
\begin{eqnarray}\label{ms}
d\big(x;F^{-1}(\oy)\big)\le\kk d\big(\oy;F(x)\big)\quad\quad\mbox{ holds for all}\quad x\in U.
\end{eqnarray}
If in addition the relationship $F^{-1}(\bar{y})\cap U=\{\bar{x}\}$ holds, then $F$ is {\sc strongly metrically subregular} with modulus $\kk>0$ at $(\ox,\oy)$.
\end{itemize}}
\end{Definition}

Metric regularity in \eqref{mr} asserts a kind of {\em Lipschitzian/linear rate} dependence of the distance of a trial point $x$ from a solution set $F^{-1}(y)$ in terms of the residual quantity $d(y;F(x))$ for all $(x,y)$ {\em around} the reference pair $(\ox,\oy)\in\gph F$. Knowledge of such a dependence is useful since this residual can in principle be easily calculated. Subregularity \eqref{ms}, a weaker notion, requires a similar estimate except the solution set $F^{-1}(\bar{y})$ remains {\em fixed}. In turn, the strong versions of metric regularity and subregularity enforce some local single-valuedness properties of the inverse mapping $F^{-1}\colon Y\tto X$. Note that the ``around" properties of metric regularity and strong metric regularity are {\em robust/stable} with respect to perturbations of the reference point while its ``at" subregularity counterparts are not. It is worth mentioning to avoid confusions that subregularity is also known (particularly in the Russian literature; see, e.g., \cite{met_ioffe}) as ``metric regularity {\em at} the point."

In our current work we consider functions taking values in the extended real line $\oR:=(-\infty,\infty]$. For such an extended-real-valued function $f:X\to\oR$ we define the {\em domain} of $f$ to be
$$
\mbox{\rm dom}\,f:=\big\{x\in X\big|\;f(x)<\infty\big\}
$$
and the {\em epigraph} of $f$ to be
$$
\mbox{\rm epi}\,f:=\big\{(x,\al)\in X\times\R\big|\;\al\ge f(x)\big\}.
$$
We always assume that functions $f\colon X\to\overline{\R}$ are {\em proper}, in the sense that they are not identically equal to $\infty$. Furthermore, unless otherwise stated, all the functions under consideration in this paper are {\em lower semicontinuous} around the reference point.

The fundamental tools for studying general nonsmooth function are subdifferentials. The following two subdifferential notions are used in this paper.

\begin{Definition}{\bf (subdifferentials of functions).}\label{subdif} {\rm Let $f\colon X\to\oR$, and let $\bar{x}\in\dom f$.
\begin{itemize}
\item  The {\sc regular/Fr\'echet subdifferential} (known also as the presubdifferential or viscosity subdifferential) of $f$ at $\ox$ is
\begin{eqnarray}\label{2.1}
\Hat\partial f(\ox):=\Big\{x^*\in X^*\Big|\;\liminf_{x\to\ox}\frac{f(x)-f(\ox)-\la x^*,x-\ox\ra}{\|x-\ox\|}\ge0\Big\}.
\end{eqnarray}
\item The {\sc limiting/Mordukhovich subdifferential} (known also as the basic or general subdifferential) of $f$ at $\ox$ is defined
via \eqref{pk} by
\begin{eqnarray}\label{2.2}
\partial f(\ox):=\Limsup_{x\st{f}{\to}\ox}\Hat\partial f(x),
\end{eqnarray}
where the symbol $x\st{f}{\to}\ox$ means that $x\to\ox$ with $f(x)\to f(\ox)$.
\end{itemize}}
\end{Definition}

For convex functions $f\colon X\to\overline{\R}$, both subdifferentials $\Hat\partial f(\ox)$ and $\partial f(\ox)$ reduce to the subdifferential of convex analysis \cite[Theorem~1.93]{M1}. On the other hand, when $f$ is ${\cal C}^1$-smooth these subdifferentials are singleton sets consisting of the Fr\'echet derivative $\nabla f(\bar{x})$.

\section{Characterizing Regularity Properties of the Subdifferential}
\setcounter{equation}{0}

In this section we relate {\em quadratic growth} properties of a function $f$ to the aforementioned regularity properties of the subdifferential $\partial f$ in \eqref{2.2}. Namely, we will see that regularity of $\partial f$ at a {\em ``near" local minimizer} always implies a quadratic growth condition on the function. The converse is not for free; it only holds under an additional {\em uniformity condition} on the subdifferentials holding, in particular, in the case of convex functions.

We begin by analyzing metric {\em subregularity} of the subdifferential. To this end, the following result complements that of \cite[Theorem~2.1]{AG_new} and answers in the {\em affirmative} the question posed in \cite[Remark~2.2(i)]{AG_new}, with the {\em sharpest possible bound} for the constant $\al$ in \eqref{3.2}. We should also mention that considering a ``near'' local minimizer in (\ref{3.1}), rather than a local minimizer, is not only for the sake of generality, but is an essential idea of the proof.

\begin{Theorem}{\rm\bf (metric subregularity of the subdifferential).}\label{thm1} Given $f\colon X\to\overline{\R}$ and a pair $(\ox,\ox^{*})\in\gph\partial f$, consider the following two statements:

{\bf (i)} The subdifferential $\partial f$ is metrically subregular at $(\ox,\ox^*)$ with modulus $\kk>0$ and there are real numbers $r\in(0,\kk^{-1})$ and $\delta>0$ such that
\begin{eqnarray}\label{3.1}
f(x)\ge f(\ox)+\la\ox^*,x-\ox\ra-\frac{r}{2}d^2\big(x;(\partial f)^{-1}(\ox^*)\big)\quad\mbox{for all}\quad x\in\B_\delta(\ox).
\end{eqnarray}

{\bf(ii)} There are real numbers $\al,\eta>0$ such that
\begin{eqnarray}\label{3.2}
f(x)\ge f(\ox)+\la\ox^*,x-\ox\ra+\frac{\al}{2}d^2\big(x;(\partial f)^{-1}(\ox^*)\big)\quad\mbox{for all}\;\;x\in\B_\eta(\ox).
\end{eqnarray}
Then implication $[{\bf (i)}\Longrightarrow{\bf (ii)}]$ holds, where $\al$ may be chosen arbitrarily in $(0,\kk^{-1})$. Furthermore, the converse implication $[{\rm\bf (ii)}\Longrightarrow{\rm\bf(i)}]$ also holds if in addition there is some $\beta\in[0,\al)$ with
\begin{eqnarray}\label{3.3}
f(u)\ge f(x)+\la x^*,u-x\ra-\frac{\beta}{2}d^2\big(x;(\partial f)^{-1}(\ox^*)\big)\;\mbox{for all}\;(u,\ox^*),(x,x^*)\in\gph\partial f\cap \B_\eta(\ox,\ox^*).
\end{eqnarray}
\end{Theorem}
{\bf Proof.} Assume first that {\bf(i)} holds. The proof proceeds by contradiction. Namely, suppose that property (\ref{3.2}) does not hold for any real numbers $\alpha,\eta >0$. Our immediate goal now is to show that there is some real $\gg>0$ such that
\begin{eqnarray}\label{3.4}
f(x)\ge f(\ox)+\la \ox^*,x-\ox\ra+\frac{1-3r\kk}{4\kk}d^2\big(x;(\partial f)^{-1}(\ox^*)\big)\quad\textrm{for all }\;x\in \B_\gg(\ox).
\end{eqnarray}
If it is not the case, we can find a sequence $\{x_k\}$ converging to $\ox$ and satisfying
\begin{eqnarray*}
f(x_k)< f(\ox)+\la\ox^*,x_k-\ox\ra+\frac{1-3r\kk}{4\kk}d^2\big(x_k;(\partial f)^{-1}(\ox^*)\big)\quad\textrm{ for each index }\;k.
\end{eqnarray*}
This together with (\ref{3.1}) implies that all $x_k$ lie outside of $(\partial f)^{-1}(\ox^*)$. Consequently we have
\begin{eqnarray*}
\inf_{x\in \B_\delta(\ox)}\Big\{f(x)+\frac{r}{2}d^2\big(x;(\partial f)^{-1}(\ox^*)\big)&-&\la \ox^*,x-\ox\ra\Big\}\ge f(\bar{x})>
\\ &>&f(x_k)+\frac{r}{2}d^2\big(x_k;(\partial f)^{-1}(\ox^*)\big)-\la\ox^*,x_k-\ox\ra-\ve_k,
\end{eqnarray*}
where $\ve_k\dn 0$ are real numbers smaller than $\frac{1-r\kk}{4\kk}d^2\big(x_k;(\partial f)^{-1}(\ox^*)\big)$. Define further the real numbers $\lm_k:=\sqrt{\frac{\kk\ve_k}{1-r\kk}}$, which also converge to zero as $k\to\infty$. Applying Ekeland's variational principle (see, e.g.,  \cite[Theorem~2.26]{M1}) ensures the existence of a new sequence \{$\hat x_k$\} satisfying $\|\hat x_k-x_k\|\le\lm_k$ and such that for each $k\in\IN$ the point $\hat x_k$ is a minimizer of the problem
\[
{\rm minimize}\quad\Big\{f(x)+\frac{r}{2}d^2\big(x;(\partial f)^{-1}(\ox^*)\big)-\la \ox^*,x-\ox\ra+\frac{\ve_k}{\lm_k}\|x-\hat x_k\|\Big\}\quad \mbox{subject to}\quad x\in\B_\delta(\ox).
\]
Since $x_k\to\ox$ and $\lm_k\dn 0$, we may suppose that $\hat x_k\in\inte\B_\delta(\ox)$ for all $k\in\IN$. Then the Fermat stationary rule along with the subdifferential sum rule from \cite[Theorem~3.41]{M1} implies the inclusions
\begin{eqnarray*}
0&\in&\partial\Big(f+\frac{r}{2}d^2\big(\cdot;(\partial f)^{-1}(\ox^*)\big)-\la\ox^*,\cdot-\ox\ra+\frac{\ve_k}{\lm_k}\|\cdot-\hat x_k\|\Big)(\hat x_k)\\
&\subset&-\bar{x}^{*}+\frac{\ve_k}{\lm_k}\B^*+\frac{r}{2}\partial
\big[d^2\big(\cdot;(\partial f)^{-1}(\ox^*)\big)\big](\hat x_k)+\partial f(\hat x_k)\\
&\subset&-\ox^*+\partial f(\hat x_k)+\Big[rd\big(\hat x_k;(\partial f)^{-1}(\ox^*)\big)+\frac{\ve_k}{\lm_k}\Big]\B^*.
\end{eqnarray*}
Combining this with the metric subregularity of $\partial f$ at $(\ox,\bar{x}^{*})$ yields the estimates
\[d\big(\hat x_k;(\partial f)^{-1}(\ox^*)\big)\le\kk d\big(\ox^*;\partial f(\hat x_k)\big)\le\kk\Big[rd\big(\hat x_k;(\partial f)^{-1}(\ox^*)\big)+\frac{\ve_k}{\lm_k}\Big]
\]
for all $k\in\IN$ sufficiently large. Hence for such $k$ the inequality
\[
(1-r\kk)d\big(\hat x_k;(\partial f)^{-1}(\ox^*)\big)\le\kk\frac{\ve_k}{\lm_k}
\]
holds. This allows us to successively deduce that
\[\begin{array}{ll}
(1-r\kk)d\big(x_k;(\partial f)^{-1}(\ox^*)\big)&\disp\le(1-r\kk)\|\hat x_k-x_k\|+(1-r\kk)d\big(\hat x_k;(\partial f)^{-1}(\ox^*)\big)\\
&\disp\le (1-r\kk)\lm_k+\kk\frac{\ve_k}{\lm_k}=\sqrt{4(1-r\kk)\kk\ve_k}\\
&\disp<(1-r\kk)d\big(x_k;(\partial f)^{-1}(\ox^*)\big),
\end{array}
\]
where the last strict inequality follows from our choice of $\ve_k$. This is a clear contradiction, which shows that inequality (\ref{3.4}) holds with some $\gg>0$. Due to our standing assumption that there are no real numbers $\al,\eta>0$ satisfying (\ref{3.2}), we conclude that
$1-3r\kk\le 0$, i.e., $r\kk\ge\frac{1}{3}$.

Define now the real number $r_1:=\frac{3r\kk-1}{2\kk}\in[0,\kk^{-1})$ and observe that inequality (\ref{3.4}) can be transformed into (\ref{3.1}) with replacing $r$ by $r_1$ and $\delta$ by $\gg$, respectively. Consequently there is some real number $\gg_1>0$ such that
\begin{eqnarray*}
f(x)\ge f(\ox)+\la\ox^*,x-\ox\ra+\frac{1-3r_1\kk}{4\kk}d^2\big(x;(\partial f)^{-1}(\ox^*)\big)\quad\textrm{ for all }\;x\in\B_{\gg_1}(\ox).
\end{eqnarray*}
As before, we obtain the inequality $r_1\kk\ge\frac{1}{3}$, or equivalently $r\kappa\ge\frac{1}{3}+\frac{2}{3^2}$. Defining further the number $r_2:=\frac{3r_1-1}{2\kk}\in[0,\kk^{-1})$ and proceeding in the same way as above lead us to the inequality $r\kk\ge\frac{1}{3}+\frac{2}{3^2}+\frac{4}{3^3}$. Then we get by induction the progressively stronger bounds
\[
r\kk\ge \frac{1}{3}+\frac{2}{3^2}+\ldots+\frac{2^k}{3^{k+1}}=1-\frac{2^{k+1}}{3^{k+1}}\quad\mbox{for all}\quad k\in\IN.
\]
Letting $k\to\infty$ yields $r\kk\ge 1$, which is a contradiction. Therefore there exist real numbers $\al,\delta>0$ such that inequality \eqref{3.2} is satisfied.

To justify {\bf (ii)}, it remains to show that $\al$ may be chosen smaller than but arbitrarily close to $\kk^{-1}$. To this end, observe that if $\al\ge \kk^{-1}$, then the claim is obvious, since it is possible to replace $\al$ by any $\al_0\in[0,\kk^{-1})$ without altering the validity of \eqref{3.2}. Hence we may suppose that $\al<\kk^{-1}$. Our immediate goal is to verify that there is some $\mu>0$ such that
\begin{eqnarray}\label{3.7}
f(x)\ge f(\ox)+\la \ox^*,x-\ox\ra+\frac{1+\al\kk}{4\kk}d^2\big(x;(\partial f)^{-1}(\ox^*)\big)\quad\textrm{ for all }\;x\in\B_\mu(\ox).
\end{eqnarray}
The process of proving this is very similar to the proof of \eqref{3.4}; hence we only sketch it. Indeed, arguing by contradiction allows us find a sequence $u_k\to\ox$ satisfying
\begin{eqnarray*}
f(u_k)<f(\ox)+\la\ox^*,u_k-\ox\ra+\frac{1+\al\kk}{4\kk}d^2\big(u_k;(\partial f)^{-1}(\ox^*)\big)\quad\textrm{ for each index }\;k.
\end{eqnarray*}
Thanks to inequality \eqref{3.2} proved above we have
\begin{eqnarray*}
\inf_{x\in\B_{\eta}(\ox)}\Big\{f(x)-\frac{\al}{2}d^2\big(x;(\partial f)^{-1}(\ox^*)\big)&-&\la\ox^*,x-\ox\ra\Big\}\ge f(\bar{x})>
\\&>&f(u_k)-\frac{\al}{2}d^2\big(u_k;(\partial f)^{-1}(\ox^*)\big)-\la\ox^*,u_k-\ox\ra-\nu_k
\end{eqnarray*}
for some $\nu_k>0$ slightly smaller than $\frac{1-\al\kk}{4\kk}d^2(u_k;(\partial f)^{-1}(\ox^*))$ and converging to zero. Define the real numbers $\rho_k:=\sqrt{\frac{\kk\nu_k}{1-\al\kk}}$, which clearly converge to zero as $k\to\infty$. By Ekeland's variational principle we find a new sequence \{$\hat u_k$\} satisfying $\|\hat u_k-u_k\|\le\rho_k$ and such that for each $k\in\IN$ the point $\hat u_k$ is a minimizer of the problem
\[
{\rm minimize}\quad\Big\{f(x)-\frac{\al}{2}d^2\big(x;(\partial f)^{-1}(\ox^*)\big)-\la\ox^*,x-\ox\ra+\frac{\nu_k}{\rho_k}\|x-\hat u_k\|\Big\}\quad \mbox{subject to}\quad x\in\B_\eta(\ox).
\]
Since $\hat u_k\to\ox$, suppose without loss of generality that $\hat u_k\in\inte\B_\eta(\ox)$ for all $k$. Applying again standard calculus rules leads us to the inclusions
\begin{eqnarray*}
0&\in&\partial\Big(f-\frac{\al}{2}d^2\big(\cdot;(\partial f)^{-1}(\ox^*)\big)-\la\ox^*,\cdot-\ox\ra+\frac{\nu_k}{\rho_k}\|\cdot-\hat u_k\|\Big)(\hat u_k)\\
&\subset&-\ox^*+ \Big[\al d\big(\hat u_k;(\partial f)^{-1}(\ox^*)\big)+\frac{\nu_k}{\rho_k}\Big]\B^*+\partial f(\hat u_k).
\end{eqnarray*}
This together with the metric subregularity of $\partial f$ at $(\ox,\ox^*)$ yields the inequality
\[
d\big(\hat u_k;(\partial f)^{-1}(\ox^*)\big)\le\kk d(\ox^*;\partial f(\hat u_k))\le\kk\Big[\al d\big(\hat u_k;(\partial f)^{-1}(\ox^*)\big)+\frac{\nu_k}{\rho_k}\Big]
\]
for all large $k$. In this way we successively deduce that
\[\begin{array}{ll}
(1-\al\kk)d\big(u_k;(\partial f)^{-1}(\ox^*)\big)&\disp\le(1-\al\kk)\|\hat u_k-u_k\|+ (1-\al\kk)d\big(\hat u_k;(\partial f)^{-1}(\ox^*)\big)\\
&\disp\le (1-\al\kk)\rho_k+\kk\frac{\nu_k}{\rho_k}= \sqrt{4(1-\al\kk)\kk\nu_k}\\
&\disp<(1-\al\kk)d\big(u_k;(\partial f)^{-1}(\ox^*)\big),
\end{array}
\]
which is a contradiction. Hence there exists $\mu>0$ such that \eqref{3.7} holds. By induction we may then construct from \eqref{3.7} two strictly positive sequences $\{\al_k\}$ and $\{\mu_k\}$ satisfying
\begin{eqnarray*}
f(x)\ge f(\ox)+\la \ox^*,x-\ox\ra+\frac{\al_k}{2}d^2\big(x;(\partial f)^{-1}(\ox^*)\big)\quad\mbox{for all}\quad x\in\B_{\mu_k}(\ox)
\end{eqnarray*}
with $\al_k:=\frac{1+\al_{k-1}\kk}{2\kk}\in(0,\kk^{-1})$ for $k\in\IN$ and $\al_0:=\al$. Consequently it gives $\al_k\to\kk^{-1}$ as $k\to\infty$ and thus completes proof of the first part of the theorem.

To justify the converse implication, assume that {\bf (ii)} holds for some $\al,\eta>0$ and pick a real number $\beta\in[0,\al)$ such that \eqref{3.3} is satisfied. We only need to prove the metric subregularity of $\partial f$ at $(\ox.\ox^*)$ since inequality \eqref{3.1} is immediate from \eqref{3.2}.  To proceed, fix any point $x\in\B_{\frac{\eta}{2}}(\ox)$ and get therefore that $d\big(x;(\partial f)^{-1}(\ox^*)\big)\le\|x-\ox\|\le \frac{\eta}{2}$. If the inequality $d\big(\ox^*;\partial f(x)\big)\ge\frac{\eta}{2}$ is valid, it clearly implies that $d\big(x;(\partial f)^{-1}(\ox^*)\big)\le d\big(\ox^*;\partial f(x))$. Hence we may suppose that $d\big(\ox^*;\partial f(x)\big)<\frac{\eta}{2}$. Taking any subgradient $x^*\in\partial f(x)\cap\B_{\frac{\eta}{2}}(\ox^*)$ yields $(x,x^*)\in\gph\partial f\cap\B_\eta(\ox,\ox^*)$. Note that for each point $u\in(\partial f)^{-1}(\ox^*)$ with $\|x-u\|\le\|x-\ox\|\le\frac{\eta}{2}$ we have $\|u-\ox\|\le\eta$. Then it follows from \eqref{3.3} that for such points $u$ the estimate
\[
f(\ox)\ge f(u)+\la\ox^*,\ox-u\ra-\frac{\beta}{2}d^2\big(u;(\partial f)^{-1}(\ox^*)\big)=f(u)+\la \ox^*,\ox-u\ra
\]
holds. Combining this with \eqref{3.2} and \eqref{3.3} gives us that
\[\begin{array}{ll}
\|x^*-\ox^*\|\cdot\|x-u\|&\disp\ge\la x^*-\ox^*,x-u\ra=\la x^*,x-u\ra+\la\ox^*,u-\ox\ra-\la\ox^*,x-\ox\ra\\
&\disp\ge f(x)-f(u)-\frac{\beta}{2}d^2\big(x;(\partial f)^{-1}(\ox^*)\big)+f(u)-f(\ox)-\la\ox^*,x-\ox\ra\\
&\disp=f(x)-f(\ox)-\la\ox^*,x-\ox\ra-\frac{\beta}{2}d^2\big(x;(\partial f)^{-1}(\ox^*)\big)\\
&\disp\ge\frac{\al-\beta}{2}d^2\big(x;(\partial f)^{-1}(\ox^*)\big).
\end{array}
\]
Since the latter inequality holds for all $u\in(\partial f)^{-1}(\ox^*)$ with $\|x-u\|\le\|x-\ox\|$, we get
\[
\|x^*-\ox^*\|\cdot d\big(x;(\partial f)^{-1}(\ox^*)\big)\ge\frac{\al-\beta}{2}d^2\big(x;(\partial f)^{-1}(\ox^*)\big),
\]
which implies in turn that
\[
\|x^*-\ox^*\|\ge\frac{\al-\beta}{2}d\big(x;(\partial f)^{-1}(\ox^*)\big)\quad\mbox{whenever}\quad x^*\in\partial f(x)\cap\B_{\frac{\eta}{2}}(\ox^*).
\]
By $d\big(\ox^*;\partial f(x)\big)\le\frac{\eta}{2}$ the latter gives us that $d\big(\ox^*;\partial f(x)\big)\ge\frac{\al-\beta}{2}d\big(x;(\partial f)^{-1}(\ox^*)\big)$, which justifies the metric subregularity of $\partial f$ at $(\ox,\ox^*)$ and thus completes the proof of the theorem. \endproof

Specializing Theorem~\ref{thm1} to the case when $\ox$ is a {\em local minimizer} of $f$ is straightforward. Note that in this case we always have $0\in\partial f(\ox)$ by the (generalized) Fermat stationary rule.

\begin{Corollary}{\bf (metric subregularity of the subdifferential at local minimizers).} Let $f\colon X\to\oR$, and let $\ox\in\dom f$ be a local minimizer of $f$. Consider the following two statements:

{\bf (i)} The subdifferential $\partial f$ is metrically subregular at $(\ox,0)$ with modulus $\kk>0$.

{\bf(ii)} There exist real numbers $\al,\eta>0$ such that
\begin{eqnarray*}
f(x)\ge f(\ox)+\frac{\al}{2}d^2\big(x;(\partial f)^{-1}(0)\big)\quad\mbox{for all}\;\;x\in \B_\eta(\ox).
\end{eqnarray*}
Then implication $[{\rm\bf(i)}\Longrightarrow{\rm\bf(ii)}]$ holds, where $\al$ may be chosen arbitrarily in $(0,\kk^{-1})$. Furthermore, the converse implication $[{\rm\bf(ii)}\Longrightarrow{\rm\bf(i)}]$ also holds if in addition there is some $\beta\in[0,\al)$ such that
$$
f(u)\ge f(x)+\la x^*,u-x\ra-\frac{\beta}{2}d^2\big(x;(\partial f)^{-1}(0)\big)\quad\mbox{for all}\quad(u,0),(x,x^*)\in\gph\partial f\cap \B_\eta(\ox,\ox^*).
$$
\end{Corollary}

It is worth mentioning that when $f$ is a convex function, inequalities \eqref{3.1} and \eqref{3.3} hold automatically. In this case, the equivalence between (i) and (ii) in the above result has been originally established by Arag\'on and Geoffroy \cite[Theorem~3.3]{AG} in Hilbert spaces and has been recently extended to Banach spaces in their new paper \cite[Theorem~2.1]{AG_new} with the weaker estimate $\al<1/2k$ in \eqref{3.2} and with the proof different from ours in Theorem~\ref{thm1}; cf. also Remark~2.2 in \cite{AG_new}.

Furthermore, it is easy to observe the following consequence of Theorem~\ref{thm1} concerning {\em strong metric subregularity} of the subdifferential and to compare it with \cite[Theorem~3.1]{AG_new} similarly to the discussion above on Theorem~\ref{thm1} and \cite[Theorem~2.1]{AG_new}.

\begin{Corollary} {\bf (strong metric subregularity of the subdifferential.)}\label{coro0} Given $f\colon X\to\oR$ and $(\ox,\ox^{*})\in\gph\partial f$, consider the following two statements:

{\bf (i)} The subdifferential $\partial f$ is strongly metrically subregular at $(\ox,\ox^*)$ with modulus $\kk>0$ and there are real numbers $r\in(0,\kk^{-1})$ and $\delta>0$ such that
\begin{eqnarray}\label{3.8}
f(x)\ge f(\ox)+\la \ox^*,x-\ox\ra-\frac{r}{2}\|x-\ox\|^2\quad\mbox{for all}\quad x\in\B_\delta(\ox).
\end{eqnarray}

{\bf(ii)} There are real numbers $\al,\eta>0$ such that
\begin{eqnarray}\label{3.9}
f(x)\ge f(\ox)+\la \ox^*,x-\ox\ra+\frac{\al}{2}\|x-\ox\|^2\quad\mbox{for all}\quad x\in\B_\eta(\ox).
\end{eqnarray}
Then implication $[{\rm\bf(i)}\Longrightarrow{\rm\bf(ii)}]$ holds, where $\al$ may be chosen arbitrarily in $(0,\kk^{-1})$. Furthermore, the converse implication $[{\rm\bf(ii)}\Longrightarrow{\rm\bf(i)}]$ also holds if in addition there is some number $\beta\in[0,\al)$ with
\begin{eqnarray}\label{3.10}
f(\ox)\ge f(x)+\la x^*,\ox-x\ra-\frac{\beta}{2}\|x-\ox\|^2\quad\mbox{whenever}\quad(x,x^*)\in\gph\partial f\cap\B_\eta(\ox,\ox^*).
\end{eqnarray}
\end{Corollary}
{\bf Proof.} Suppose first that {\bf(i)} holds. By Theorem~\ref{thm1} it suffices to show that the equality $d\big(x;(\partial f)^{-1}(\ox^*)\big)=\|x-\ox\|$ is satisfied for all $x$ around $\ox$. But this follows immediately from the assumed strong metric subregularity of $\partial f$  at $(\ox,\ox^*)$. Hence property {\bf(ii)} is verified.

To justify the converse implication, suppose that both inequalities \eqref{3.9} and \eqref{3.10} are satisfied for some constants $\al>\beta>0$ and $\eta>0$.  We now show that $(\partial f)^{-1}(\ox^*)\cap\B_\eta(\ox)=\{\ox\}$.  To proceed, pick any point $u\in(\partial f)^{-1}(\ox^*)\cap \B_\eta(\ox)$ and derive from relationships \eqref{3.9} and \eqref{3.10} that
\[
\la\ox^*,u-\ox\ra+\frac{\beta}{2}\|u-\ox\|^2\ge f(u)-f(\ox)\ge\la\ox^*,u-\ox\ra+\frac{\al}{2}\|u-\ox\|^2.
\]
Combining this with the assumed inequality $\beta<\al$ ensures that $u=\ox$. Then applying Theorem~\ref{thm1} completes the proof of this corollary.\endproof

As the following example illustrates, implication [{\bf(ii)}$\Longrightarrow${\bf(i)}] in Corollary~\ref{coro0} may easily fail if condition \eqref{3.10} is not satisfied.

\begin{Example}{\bf (second-order growth may not imply strong subregularity).}
{\rm Consider the l.s.c.\ function $f:\R\to\oR$ defined by
\[
f(x):=\left\{\begin{array}{ll} \frac{1}{2}x-x^2\sin\Big(\frac{1}{x}\Big)\quad&\mbox{if}\quad x>0, \\
0\quad&\mbox{if}\quad x=0,\\
\infty&\mbox{if}\quad x<0.
\end{array}\right.
\]
It is clear that $f$ satisfies inequality \eqref{3.9} with $\ox=0$, $\ox^*=0$, and $\al=1$. On the other hand, $\ox^*$ is not an isolated point of $\partial f^{-1}(\ox)$. Consequently assertion (i) is violated. The reason is of course that condition \eqref{3.10} does not hold for this function.}
\end{Example}

When $\ox$ is a {\em local minimizer} of $f$, we immediately deduce the following consequence of Corollary~\ref{coro0} in terms of the classical quadratic growth condition; see, e.g., \cite{BS}.

\begin{Corollary}{\bf (strong metric subregularity of the subdifferential at minimizers).} Let $f\colon X\to\R$, and let $\ox\in\dom f$ be a local minimizer of $f$. Consider the following two statements:

{\bf (i)} The subdifferential $\partial f$ is strongly metrically subregular at $(\ox,0)$ with modulus $\kk>0$.

{\bf(ii)} There are real numbers $\al,\eta>0$ such that
\begin{eqnarray*}
f(x)\ge f(\ox)+\frac{\al}{2}\|x-\ox\|^2\quad\mbox{for all}\quad x\in\B_\eta(\ox).
\end{eqnarray*}
Then implication $[{\rm\bf(i)}\Longrightarrow{\rm\bf(ii)}]$ holds, where $\al$ may be chosen arbitrarily in $(0,\kk^{-1})$. Furthermore, the converse implication $[{\rm\bf(ii)}\Longrightarrow{\rm\bf(i)}]$ also holds if in addition there is some $\beta\in[0,\al)$ with
$$
f(\ox)\ge f(x)+\la x^*,\ox-x\ra-\frac{\beta}{2}\|x-\ox\|^2\quad\mbox{whenever}\quad(x,x^*)\in\gph\partial f\cap\B_\eta(\ox,\ox^*).$$
\end{Corollary}\vspace*{0.05in}

It is tempting now to adapt the proof technique of Theorem~\ref{thm1} in order to relate quadratic growth properties of a function $f$ to {\em metric regularity} of the subdifferential $\partial f$. To do so, it seems necessary to consider a {\em uniform version} of equation (\ref{3.1}). However, it turns out to be problematic from the following two perspectives. Firstly, the case when a point $\bar{x}$ is only a local minimizer seems to be out of reach and, secondly, this seemingly weak condition when combined with metric regularity of $\partial f$ already implies that the subdifferential mapping is {\em strongly} metrically regular; cf.\ \cite{MN} and the discussion below.

\begin{Lemma}{\bf (growth conditions under metric regularity).}\label{lem:equiv} Given $f\colon X\to\oR$, fix a pair $(\ox,\ox^*)\in\gph\partial f$ and suppose that the subdifferential $\partial f$ is metrically regular around $(\ox,\ox^*)$ with modulus $\kk>0$. Then the following assertions are equivalent:

{\bf (i)} There exist a real number $r\in[0,\kk^{-1})$ and neighborhoods $U$ of $\ox$ and $U^*$ of $\ox^*$ such that for any pair $(u,u^*)\in\gph \partial f\cap(U\times U^*)$ we have
\begin{eqnarray}\label{3.11}
f(x)\ge f(u)+\la u^*,x-u\ra-\frac{r}{2}d^2\big(x;(\partial f)^{-1}(u^*)\big)\quad \mbox{whenever}\quad x\in U.
\end{eqnarray}

{\bf (ii)} There are neighborhoods $U_0$ of $\ox$ and $U^*_0$ of $\ox^*$ such that
\begin{eqnarray}\label{3.13}
f(x)\ge f(u)+\la u^*,x-u\ra\quad\quad\mbox{for all}\quad x\in U_0\;\textrm{ and }\;(u,u^*)\in\gph\partial f\cap(U_0\times U^*_0).
\end{eqnarray}
\end{Lemma}
\begin{proof} Implication [{\bf(ii)}$\Longrightarrow${\bf(i)}] holds trivially. Hence we focus on [{\bf(i)}$\Longrightarrow${\bf(ii)}], which is rather similar to the one of Theorem~\ref{thm1}. Arguing by contradiction, suppose that assertion {\bf(ii)} does not hold and then claim that there are neighborhoods $U_1$ of $\ox$ and $U^*_1$ of $\ox^*$ such that for any pairs $(u,u^*)\in\gph\partial f\cap(U_1\times U^*_1)$ we have the inequality
\begin{eqnarray}\label{3.14}
f(x)\ge f(u)+\la u^*,x-u\ra+\frac{1-3r\kk}{4\kk}d^2\big(x;(\partial f)^{-1}(u^*)\big)\quad\textrm{whenever}\quad x\in U_1.
\end{eqnarray}
Assuming on the contrary that this condition is not satisfied, we find sequences $\{x_k\}$ converging to $\ox$ and $\{(u_k,u^*_k)\}\subset\gph\partial f$ converging to $(\ox,\ox^*)$ such that
\begin{eqnarray*}
f(x_k)< f(u_k)+\la u^*_k,x_k-u_k\ra+\frac{1-3r\kk}{4\kk}d^2\big(x_k;(\partial f)^{-1}(u^*_k)\big)\quad\textrm{ for each index }\;k.
\end{eqnarray*}
This together with (\ref{3.11}) implies that $x_k\notin(\partial f)^{-1}(u_k^{*})$ and consequently the inequalities
\begin{eqnarray*}
\inf_{x\in\B_\delta(\ox)}\Big\{f(x)+\frac{r}{2}d^2\big(x;(\partial f)^{-1}(u_k^*)\big)&-&\la u_k^*,x-u_k\ra\Big\}\ge f(u_k)>\\&>& f(x_k)+\frac{r}{2}d^2\big(x_k;(\partial f)^{-1}(u_k^*)\big)-\la u_k^*,x_k-u_k\ra-\ve_k
\end{eqnarray*}
hold, where $\ve_k>0$ are slightly smaller than $\frac{1-r\kappa}{4\kappa}d^2(x_k;(\partial f)^{-1}(u_k^{*}))$. Without loss of generality assume that $\ve_k$ converge to zero. Define also the real numbers $\lm_k:=\sqrt{\frac{\kk\ve_k}{1-r\kk}}$, which clearly also converge to zero.
Employing Ekeland's variational principle, we deduce the existence of a new sequence \{$\hat x_k$\} satisfying $\|\hat x_k-x_k\|\le\lm_k$ and such that each $\hat x_k$ is a minimizer of the problem
\[
{\rm minimize}\quad \Big\{f(x)+\frac{r}{2}d^2(x;(\partial f)^{-1}(u^*_k))-\la u^*_k,x-u_k\ra+\frac{\ve_k}{\lm_k}\|x-\hat x_k\|\Big\}\quad \mbox{subject to}\quad x\in\B_\delta(\ox).
\]
Since $x_k\to\ox$ and $\hat x_k\to\ox$, we may suppose $\hat x_k\in{\rm int}\,\B_\delta(\ox)$ for all $k\in\IN$. Applying then the Fermat stationary rule along with standard subdifferential calculus ensures the inclusions
\begin{eqnarray*}
0&\in&\partial\Big(f+\frac{r}{2}d^2\big(\cdot;(\partial f)^{-1}(u_k^*)\big)-\la u_k^*,\cdot-u_k\ra+\frac{\ve_k}{\lm_k}\|\cdot-\hat x_k\|\Big)(\hat x_k)\\
&\subset&-u_k^*+\partial f(\hat x_k)+\Big[rd\big(\hat x_k;(\partial f)^{-1}(u_k^*)\big)+\frac{\ve_k}{\lm_k}\Big]\B^*.
\end{eqnarray*}
This together with the assumed metric regularity of $\partial f$ around $(\ox,\ox^*)$ yields the inequality
\[
d\big(\hat x_k;(\partial f)^{-1}(u^*_k)\big)\le\kk d\big(u_k^*;\partial f(\hat x_k)\big)\le\kk\Big[rd(\hat x_k;(\partial f)^{-1}(u_k^*)\big)+\frac{\ve_k}{\lm_k}\Big]
\]
for all $k\in\IN$ sufficiently large. Hence for such $k$ the inequality
$$
(1-r\kk)d\big(\hat{x}_k;(\partial f)^{-1}(u_k^*)\big)\disp\le\kappa\frac{\ve_k}{\lambda_k}
$$
is satisfied. It allows us to deduce successively that
\[\begin{array}{ll}
(1-r\kk)d\big(x_k;(\partial f)^{-1}(u_k^*)\big)&\disp\le(1-r\kk)\|\hat x_k-x_k\|+(1-r\kk)d\big(\hat x_k;(\partial f)^{-1}(u_k^*)\big)\\
&\disp\le (1-r\kk)\lm_k+\kk\frac{\ve_k}{\lm_k}=\sqrt{4(1-r\kk)\kk\ve_k}\\
&\disp<(1-r\kk) d\big(x_k;(\partial f)^{-1}(u_k^*)\big),
\end{array}
\]
where the last inequality follows from our choice of $\ve_k$.  This is clearly a contradiction, which proves the existence of neighborhood $U_1$ of $\ox$ and $U^*_1$ of $\ox^*$ such that inequality (\ref{3.14}) holds and so condition \eqref{3.13} is not satisfied. Thus we get the inequality $1-3r\kk\le 0$, or equivalently $r\kk\ge\frac{1}{3}$.

To proceed further, define the real number $r_1:=\frac{3r\kk-1}{2\kk}\in[0,\kk^{-1})$ and observe that inequality \eqref{3.14} can be transformed into \eqref{3.11} with replacing $r$ by $r_1$, $U$ by $U_1$, and $U^*$ by $U^*_1$, respectively. This gives us $r_1\kk\ge\frac{1}{3}$, or equivalently $r\ge\frac{1}{3}+\frac{2}{3^2}$. Defining now $r_2:=\frac{3r_1-1}{2\kk}\in[0,\kk^{-1})$ and proceeding in the same way lead us to the inequality $r\kk\ge\frac{1}{3}+\frac{2}{3^2}+\frac{4}{3^2}$. Finally, by induction we arrive at the progressively stronger bounds
\[
r\kk\ge\frac{1}{3}+\frac{2}{3^2}+\ldots+\frac{2^k}{3^{k+1}}=1-\frac{2^{k+1}}{3^{k+1}}\quad\mbox{for all}\quad k\in\IN.
\]
Letting $k\to\infty$ yields $r\kappa\ge 1$, which is a clear contradiction that completes the proof.
\end{proof}

Employing this lemma, we now establish some characterizations of {\em metric regularity} and {\em strong metric regularity} of the subdifferential for functions on Asplund spaces. Though parts of the next theorem may be found in \cite{MN}, it nicely illustrates the techniques developed in the current work.

\begin{Theorem}{\bf (metric regularity and strong metric regularity of the subdifferential).}\label{thm2} Given $f\colon X\to\oR$ and $(\ox,\ox^*)\in\gph\partial f$, the following assertions are equivalent:

{\bf (i)} The subdifferential $\partial f$ is metrically regular around $(\ox,\ox^*)$ with modulus $\kk>0$ and there exist a real number
$r\in[0,\kk^{-1})$ and neighborhoods $U$ of $\ox$ and $U^*$ of $\ox^*$ such that for any pair $(u,u^*)\in\gph\partial f\cap (U\times U^*)$ we have
condition \eqref{3.11}.

{\bf(ii)} There exist neighborhoods $V$ of $\ox$ and $V^*$ of $\ox^*$ such that for any  $u^*\in V^*$ there is a point $u\in(\partial f)^{-1}(u^*)\cap V$ satisfying the inequality
\begin{eqnarray}\label{3.12}
f(x)\ge f(u)+\la u^*,x-u\ra+\frac{1}{2\kk}\|x-u\|^2\quad\mbox{whenever}\quad x\in V.
\end{eqnarray}

{\bf (iii)} The subdifferential $\partial f$ is metrically regular around $(\ox,\ox^*)$ with modulus $\kk>0$ and there are neighborhoods $W$ of $\ox$ and $W^*$ of $\ox^*$ such that
\begin{eqnarray*}
f(x)\ge f(u)+\la u^*,x-u\ra\quad\quad\mbox{for all}\quad x\in U\;\textrm{ and }\;(u,u^*)\in\gph\partial f\cap(W\times W^*).
\end{eqnarray*}

{\bf(iv)} The point $\bar{x}$ is a local minimizer of the function $x\mapsto f(x)-\langle \bar{x}^{*},x\rangle$ and the subdifferential $\partial f$ is strongly metrically regular around $(\bar{x},\bar{x}^*)$ with modulus $\kk$.
\end{Theorem}
{\bf Proof.} Observe first that implication [{\bf(i)}$\Longrightarrow${\bf(iii)}] immediately follows from Lemma~\ref{lem:equiv}. Now we show that [{\bf(iii)}$\Longrightarrow${\bf(iv)}]. Suppose to this end that assertion {\bf(iii)} holds and let $T$ be a localization of $\partial f$ with respect to $W$ and $W^*$. It is easy to check that the operator $T$ is monotone, i.e.,
\[
\la x^*-u^*,x-u\ra\ge 0\quad\mbox{whenever}\quad(x,x^*),\,(u,u^*)\in\gph T.
\]
Since $\partial f$ is metrically regular around $(\ox,\ox^*)$, we may assume that $T$ also has this property around $(\ox,\ox^*)$. Taking into account that metric regularity of a mapping is equivalent to the Lipschitz-like/Aubin property of the inverse (and hence it is lower/inner semicontinuous), the local monotonicity of $T^{-1}$ implies its local {\em single-valuedness} around $\ox$ by the classical Kenderov theorem \cite{k}. Thus $\partial f$ is strongly metrically regular around $(\ox,\ox^*)$, which verifies {\bf(iv)}. Finally, implication [{\bf(iv)}$\Longrightarrow${\bf(ii)}] follows from \cite[Theorem~3.2]{MN} while [{\bf(ii)}$\Longrightarrow${\bf(i)}] holds trivially.\endproof

When $X$ is a Hilbert space, the following consequence of Theorem~\ref{thm2} shows that the estimate in \eqref{3.11} can be replaced by a more direct and seemingly weaker condition.

\begin{Corollary}{\bf (metric regularity and strong metric regularity of the subdifferential in Hilbert spaces).}\label{coro1} Let $X$ be a Hilbert space in the setting of Theorem~{\rm\ref{thm2}}. Then all the assertions of this theorem are equivalent to the following one:

{\bf (v)} The subgradient mapping $\partial f$ is metrically regular around $(\ox,\ox^*)$ with modulus $\kk>0$ and there are some $r\in[0,\kk^{-1})$ and neighborhoods $U$ of $\ox$, $U^*$ of $\ox^*$ such that
\begin{eqnarray}\label{3.15}
f(x)\ge f(u)+\la u^*,x-u\ra-\frac{r}{2}\|x-u\|^2\quad\mbox{for all}\quad x\in U,\;(u,u^*)\in\gph\partial f\cap(U\times U^*).
\end{eqnarray}
\end{Corollary}
{\bf Proof.} It is easy to see that assertion {\bf(i)} of Theorem~\ref{thm2} implies {\bf(v)}. Thus we only need to verify that {\bf(v)} implies assertion {\bf(iii)} therein. To this end, suppose that {\bf(v)} holds while {\bf(iii)} does not. Define the function $g(x):=f(x)+\frac{r}{2}\|x-\ox\|^2$ and observe that $\partial g=\partial f+r(I-\ox)$, where $I:X\to X$ is the identity mapping. Define further $J(x,y):=(x,y+r(x-\ox))$ for $(x,y)\in X\times X$ and deduce from the classical open mapping theorem that $W:=J(U\times U^*)$ is a neighborhood of $(\ox,\ox^*)$. Using \eqref{3.15} ensures that
\begin{eqnarray*}
g(x)\ge g(u)+\la u^*,x-u\ra\quad\mbox{for all}\quad x\in U\;\textrm{ and }\;(u,u^*)\in\gph\partial g\cap W.
\end{eqnarray*}
Furthermore, it follows from \cite[Theorem~3.3]{DLR}  that the subdifferential mapping $\partial g$ is metrically regular around $(\ox,\ox^*)$ with radius $\frac{\kk}{1-r\kk}$. Employing now Theorem~\ref{thm2} for the function $g$ allows us to find neighborhoods $W_1$ of $\ox$ and $W^*_1$ of $\ox^*$ satisfying
\[
g(x)\ge g(u)+\la u^*,x-u\ra+\frac{1-r\kk}{2\kk}\|x-u\|^2\quad\mbox{for all}\quad x\in W_1,(u,u^*)\in\gph\partial g\cap(W_1\times W^*_1).
\]
Hence there are neighborhoods $U_1$ of $\ox$ and $U^*_1$ of $\ox^*$ with $U_1\times U^*_1\subset J^{-1}(W_1\times W^*_1)$ such that for any point $x\in U_1$ and any pair $(u,u^*)\in\gph\partial f\cap(U_1\times U_1^*)$ we have
\begin{eqnarray}
f(x)&\ge& f(u)+\la u^*,x-u\ra+\frac{1-r\kk}{2\kk}\|x-u\|^2+\frac{r}{2}\Big[\|u-\ox\|^2-\|x-\ox\|^2+2\la u-\ox,x-u\ra\Big]\nonumber\\
&\ge& f(u)+\la u^*,x-u\ra+\frac{1-r\kk}{2\kk}\|x-u\|^2+\frac{r}{2}\Big[\la u-x, u+x-2\ox\ra+2\la u-\ox,x-u\ra\Big]\nonumber\\
&\ge& f(u)+\la u^*,x-u\ra+\frac{1-r\kk}{2\kk}\|x-u\|^2-\frac{r}{2}\|u-x\|^2\nonumber\\
&\ge& f(u)+\la u^*,x-u\ra+\frac{1-2r\kk}{2\kk}\|x-u\|^2.\label{3.16}
\end{eqnarray}
Consequently it ensures that $1-2r\kk\le 0$, or equivalently $r\kk\ge\frac{1}{2}$. Defining $r_1:=\frac{2r\kk-1}{\kk}\in[0,\kk^{-1})$, we rewrite \eqref{3.16} in form \eqref{3.15} by replacing $r$ with $r_1$. This gives us in turn that $r_1\kk\ge\frac{1}{2}$, or equivalently $r\kk\ge\frac{1}{2}+\frac{1}{2^2}$. By induction we deduce the estimate
\[
r\kk\ge\frac{1}{2}+\frac{1}{2^2}+\ldots+\frac{1}{2^k}=1-\frac{1}{2^k}\quad\mbox{for all}\quad k\in\IN.
\]
Letting $k\to\infty$ yields $r\kk\ge 1$, which is a clear contradiction that thus justifies {\bf(iii)}.\endproof

\section{New Second-Order Characterizations of Tilt Stability}
\setcounter{equation}{0}

In this section we continue the study of metric regularity and strong metric regularity of the limiting subdifferential $\partial f$ with the goal of studying their relationships with the fundamental notion of {\em tilt stability} in optimization, formulated in Definition~\ref{tilt}. To motivate the subsequent discussion, consider first a ${\cal C}^2$-smooth function $f\colon\R^n\to\R$ and its {\em local minimizer} $\bar{x}$. It is rather standard to deduce from classical results of nonlinear analysis that the following conditions are {\em equivalent}:\vspace*{0.05in}

{\bf (a)} The gradient mapping $\nabla f\colon\R^n\to\R^n$ is metrically regular around $(\bar{x},0)$.\vspace*{0.05in}

{\bf (b)} The gradient mapping $\nabla f$ is strongly metrically regular around $(\bar{x},0)$.\vspace*{0.05in}

{\bf (c)} The Hessian matrix $\nabla^2f(\bar{x})\in\R^n\times\R^n$ is positive-definite.\vspace*{0.05in}

{\bf (d)} ${\rm Ker}\,\nabla^2f(\bar{x})=\{0\}$ for the Hessian kernel ${\rm Ker}\,\nabla^2 f(\bar{x}):=\{u\in\R^n|\;\nabla^2 f(\ox)u=0\}$.\\[1ex]
It is worth recalling to this end that condition ${\bf (c)}$ is {\em sufficient} for local optimality of $\ox$ while the positive-semidefiniteness of $\nabla^2 f(\ox)$ is {\em necessary} for $\ox$ to be a local minimizer of $f$.\vspace*{0.05in}

In what follows we would like to analyze to which extend the above equivalences between conditions ${\bf (a)}$--${\bf (d)}$ for local minimizers carry over to the nonsmooth setting. To proceed, let us define first appropriate second-order generalized differential constructions playing the role of the Hessian. We begin with generalized normals to sets.

Given a nonempty subset $\O\subset X$ of an Asplund space, we associate with it the {\em indicator function} $\delta_\O\colon X\to\oR$ by setting it equal to zero on $\O$ and to $\infty$ otherwise. Geometric analogues of the subdifferentials \eqref{2.1} and \eqref{2.2}
are the corresponding {\em normal cones} given by
\begin{eqnarray}\label{2.3}
\Hat N_\O(\ox):=\Hat\partial\delta_\O(\ox)\quad\mbox{ and }\quad N_\O(\ox):=\partial\delta_\O(\ox),
\end{eqnarray}
which relate to each other by \eqref{2.2} via the outer limit
\begin{equation}\label{2.3a}
N_\O(\ox)=\Limsup_{x\st{\O}{\to}\ox}\Hat N_\O(x),
\end{equation}
where $x\st{\O}{\to}\ox$ signifies the convergence $x\to\ox$ with $x\in\O$.

Next we recall the {\em coderivative} constructions for set-valued mappings between Asplund spaces defined via the normal cones \eqref{2.3} to their graphs.
\begin{Definition}{\bf (coderivatives of mappings).} {\rm Let $F\colon X\tto Y$ be a set-mapping with nonempty domain. The {\sc regular} and {\sc limiting coderivatives} of $F$ at $(\bar{x},\bar{y})\in\gph F$ are set-valued mappings from $Y^{*}$ to $X^{*}$ given, respectively, by
\begin{eqnarray}\label{2.4}
\Hat D^*F(\ox,\oy)(y^*):=\big\{x^*\in X^*\big|\;(x^*,-y^*)\in\Hat N_{\gph F}(\ox,\oy)\big\}\;\mbox{ for all }\;y^*\in Y^*,
\end{eqnarray}
\begin{eqnarray}\label{2.5}
D^*F(\ox,\oy)(y^*):=\big\{x^*\in X^*\big|\;(x^*,-y^*)\in N_{\gph F}(\ox,\oy)\big\}\;\mbox{ for all }\;y^*\in Y^*,
\end{eqnarray}
where we omit $\oy=F(\ox)$ for single-valued mappings.}
\end{Definition}

The coderivative constructions in (\ref{2.4}) and (\ref{2.5}) have proven themselves to be useful tools for the study and characterizations of well-posedness and other issues in variational analysis and optimization; see, e.g., \cite{M1,RW} for more details and references. In particular,
applying the coderivative \eqref{2.5} to the subdifferential mapping $\partial f\colon X\tto X^*$ for an extended-real-valued function $f$ yields the
{\em second-order subdifferential} (or {\em generalized Hessian}) of $f$ at $(\ox,\ox^*)\in\gph\partial f$ introduced in \cite{M0} by
\begin{eqnarray}\label{2.10}
\partial^2f(\ox,\ox^*)(u):=\big(D^*\partial f\big)(\ox,\ox^*)(u)\quad\mbox{ for all }\;u\in X^{**},
\end{eqnarray}
which plays a crucial role in the subsequent results of this section. Note that
\begin{equation}\label{c2}
\partial^2 f(\ox)(u)=\big\{\nabla^2 f(\ox)^*u\}\;\mbox{ for all }\;u\in X^{**}
\end{equation}
if $f\in{\cal C}^2$, where $\nabla^2f(\ox)^*=\nabla^2f(\ox)$ in the case of Hilbert spaces $X$. The first remarkable fact about the generalized Hessian \eqref{2.10} is that in its terms we have the full counterpart of the equivalence between the above conditions {\bf (a)} and {\bf (d)} in the {\em general nonsmooth case} of $f\colon\R^n\to\oR$, where the extension of {\bf (d)} is written as
\begin{equation}\label{ker}
{\rm Ker}\,\partial^2 f(\ox,\ox^*)=\{0\}\;\mbox{ with }\;{\rm Ker}\,\partial^2f(\ox,\ox^*):=\big\{u\in\R^n\big|\;0\in\partial^2 f(\ox,\ox^*)(u)\big\}.
\end{equation}
Characterization \eqref{ker} of the metric regularity of $\partial f$ does not even require $\ox$ to be a local minimizer of $f$ while only imposing the closedness assumption on $\gph\partial f$, which is always the case when $f$ is either continuous around $\ox$ or it belongs to favorable classes of extended-real-valued l.s.c.\ functions overwhelmingly encountered in variational analysis and optimization; see, e.g., \cite{M1,RW} for more details. This result follows from the more general characterization ${\rm Ker}\,D^*F(\ox,\oy)=\{0\}$ of metric regularity of closed-graph mappings $F\colon\R^n\tto\R^m$ around $(\ox,\oy)\in\gph F$ known as the coderivative or Mordukhovich criterion; see \cite[Corollary~4.3]{mor93}, \cite[Theorem~9.43]{RW}, and the references therein. The Asplund space extension of this criterion is given in the book \cite[Theorem~4.18]{M1}, where the reader can find further references and discussions.

Another important application of the generalized second-order construction \eqref{2.10} is that replacing $\nabla^2f(\ox)$ by $\partial^2f(\ox,\ox^*)$ in the positive-definiteness condition {\bf (c)} provides a {\em characterization} of {\em tilt-stable} local minimizers of $f$ for a broad class of extended-real-valued functions. To proceed, we need to recall the notion of prox-regularity introduced in \cite{PR1}.

\begin{Definition}{\bf (prox-regularity of functions).}\label{prox-reg} {\rm A function $f:X\to\oR$ is {\sc prox-regular} at $\ox$ for $\ox^*\in\partial f(\ox)$ if $f$ is finite at $\bar{x}$ and there are constants $r>0$ and $\ve>0$ such that for all $x,u\in\B_\ve(\ox)$ with $|f(u)-f(\ox)|\le\ve$ we have
\begin{eqnarray}\label{2.8}
f(x)\ge f(u)+\la u^*,x-u\ra-\frac{r}{2}\|x-u\|^2\quad\mbox{ whenever }\;u^*\in\partial f(u)\cap\B_\ve(\ox^*).
\end{eqnarray}}
\end{Definition}

The next definition asserts a kind of continuity of $f$ relative to convergence on $\gph\partial f$.

\begin{Definition}{\bf (subdifferential continuity of functions).} {\rm A function $f\colon X\to\overline{\R}$ is {\sc subdifferentially continuous} at $\bar{x}$ for $\bar{x}^*\in\partial f(\bar{x})$ if the mapping $(x,x^*)\mapsto f(x)$ is continuous at $(\ox,\ox^*)$ relative to the subdifferential graph $\gph\partial f$ in the norm$\times$norm topology.}
\end{Definition}

When $f$ is both prox-regular and subdifferentially continuous at $\ox$ for $\ox^*\in\partial f(\ox)$, it is easy to see that the condition ``$|f(u)-f(\ox)|\le\ve$" can be omitted in the definition of prox-regularity. In this case the {\em subdifferential graph} $\gph\partial f$ is {\em locally closed} around $(\ox,\ox^*)$ in the norm$\times$norm topology. It is well known from \cite{BT2,BT1,PR1,RW} that the class of prox-regular and subdifferentially continuous functions is rather broad including, in particular, {\em strongly amenable} functions \cite{RW} on $\R^n$ and all the l.s.c.\ convex functions on arbitrary Banach spaces.\vspace*{0.05in}

Now we ready to recall a major notion of tilt-stable minimizers introduced in \cite{PR2}.

\begin{Definition}{\bf (tilt stability of locally optimal solutions).}\label{tilt} {\rm Given $f\colon X\to\oR$, we say that a point $\ox\in\dom f$ is a {\sc tilt-stable local minimizer} of $f$ if there is $\gg>0$ such that the mapping
\begin{eqnarray}\label{2.9}
M_\gg:x^*\mapsto{\rm argmin}\big\{f(x)-\la x^*,x\ra\big|\;x\in\B_\gg(\ox)\big\}
\end{eqnarray}
is single-valued and Lipschitz continuous on some neighborhood of $0\in X^*$ with $M_\gg(0)=\ox$.}
\end{Definition}

We also consider a {\em quantitative} version of this notion by specifying the modulus of tilt stability. Namely, $\ox$ is a tilt-stable local minimizer of $f$ with {\em modulus} $\kk>0$ if the mapping $M_\gg$ in \eqref{2.9} is single-valued with $M_\gg(0)=\ox$ and Lipschitz continuous with constant $\kk$.\vspace*{0.05in}

The next result shows the close relationship between tilt-stability and strong metric regularity of the subdifferential. Finite-dimensional versions of this relationship (without modulus part) were revealed in \cite[Proposition~7.2]{LZ} and \cite[Theorem~3.3]{DL} while its Asplund space counterpart with Hilbert space specification were given in \cite[Theorem~4.2]{MN}.

\begin{Proposition}{\bf (tilt-stability via strong metric regularity of the subdifferential in Hilbert spaces).}\label{pro1} Let $X$ be Hilbert, and let $f\colon X\to\oR$ be subdifferentially continuous at $\ox\in\dom f$ for $\ox^*=0\in\partial f(\bar{x})$. Then all the assertions of Corollary~{\rm\ref{coro1}} are equivalent to the following:

{\bf (vi)} $f$ is prox-regular at $\ox$ for $\ox^*$ and $\ox$ is a tilt-stable local minimizer of $f$ with modulus $\kk$.
\end{Proposition}
{\bf Proof.} This can be easily deduced from \cite[Theorem~4.2]{MN} and Corollary~\ref{coro1}.\endproof

The main result of Poliquin and Rockafellar \cite[Theorem~1.3]{PR2} is the following second-order subdifferential characterization of tilt stability in finite-dimensional spaces. Note that it shows, in particular, that for ${\cal C}^2$-smooth functions $f\colon\R^n\to\R$ the positive-definiteness of the Hessian $\nabla^2 f(\ox)$ is {\em necessary and sufficient} for $\ox$ being a tilt-stable minimizer of $f$. As we see below, the main result of \cite{PR2} is implied by our more general results obtained in this section (see Theorem~\ref{coro5}) with a more direct and simple proof in comparison with the original one in \cite[Theorem~1.3]{PR2}.

\begin{Theorem} {\bf(tilt stability via the generalized Hessian in finite dimensions).}\label{PR} Consider $f\colon\R^n\to\oR$ with $\ox^*=0\in\partial f(\ox)$ and assume that $f$ is both prox-regular and subdifferentially continuous at $\ox$ for $\ox^*$. Then we have the equivalent assertions:

{\bf (i)} The point $\ox$ is a tilt-stable local minimizer of the function $f$.

{\bf(ii)} The generalized Hessian $\partial^2f(\ox,0)$ is positive-definite in the sense that
\begin{eqnarray}\label{4.11a}
\la u^*,u\ra>0\quad\mbox{whenever}\quad u^*\in\partial^2f(\ox,0)(u),\;u\ne 0.
\end{eqnarray}
\end{Theorem}

Theorem~\ref{PR} together with Proposition~\ref{pro1} sheds light on an appropriate nonsmooth extension of the aforementioned equivalence ${\bf(b)}\Longleftrightarrow{\bf(c)}$ between strong metric regularity of the subdifferential and positive-definiteness of the Hessian for ${\cal C}^2$-smooth functions. Following this line, it is interesting to compare {\em metric regularity} and {\em strong metric regularity} of the subdifferential. As we know, these two notions are the same for ${\cal C}^2$-smooth functions. When the function $f$ is convex, the subdifferential $\partial f$ is maximal monotone, and consequently the two notions are the same yet again by the Kenderov theorem; see the proof of Theorem~\ref{thm2}. The answer is also positive for a broad class of functions given by $f(x)=f_0(x)+\delta_\O(x)$, where $f_0$ is a ${\cal C}^2$-smooth function and $\O$ is a polyhedral convex set in finite dimensions. Indeed, this follows from \cite[Theorem~3]{DR1}.

Based on the above discussion, it is tempting to {\em conjecture} that metric regularity and strong metric regularity of the subdifferential are the {\em same notions} for every prox-regular and subdifferentially continuous function on a finite-dimensional space. We leave this as a conjecture, which we plan to investigate in future work.
\begin{Conjecture}[equivalence of metric regularity and strong metric regularity]\label{conj} {\hfill\hfill \\}
Consider a function $f\colon\R^n\to\overline{\R}$ that is both prox-regular and subdifferentially continuous at $\bar{x}$ for $\bar{x}^*=0$, where $\bar{x}$ is a local minimizer of $f$. Then $\partial f$ is metrically regular at $(\bar{x},\bar{x}^{*})$ if and only if $\partial f$ is strongly metrically regular at $(\bar{x},\bar{x}^{*})$.
\end{Conjecture}

\begin{Remark}
{\rm Conjecture~\ref{conj} is certainly false outside the class of prox-regular subdifferentially continuous functions. For example, consider the function $f(x_1,x_2)=x_1^2+x_2^2 +\delta_{\{(x,y)| xy=0\}}(x_1,x_2)$ on $\R^2$. Define $\bar{x}:=(0,0)$ and $\bar{x}^*:=(0,0)$. It's easy to see that $f$ is metrically regular at $(\bar{x},\bar{x}^{*})$ but is not strongly metrically regular there. The trouble is that $f$ is not prox-regular at $\bar{x}$ for $\bar{x}^{*}$.}
\end{Remark}

If Conjecture~\ref{conj} were true, then many known descriptions of metric regularity would be very helpful in studying and characterizing tilt stability of local minimizers. Even though we are unable to establish the validity of this conjecture, we are able to prove in our major Theorem~\ref{coro5} that metric regularity of the subdifferential combined with {\em positive-semidefiniteness} of the generalized Hessian indeed characterizes strong metric regularity of the subdifferential at local minimizers of prox-regular subdifferentially continuous functions on $\R^n$.\vspace*{0.05in}

To proceed in this direction, we derive now new generalized differential characterizations of tilt-stable local minimizers in the Hilbert space setting. The following second-order construction introduced in \cite{MN} is crucial for this purpose in infinite dimensions and also to obtain {\em quantitative} characterizations of tilt-stable minimizers for nonsmooth functions in finite-dimensional spaces.

\begin{Definition}{\bf (combined second-order subdifferential).}\label{2nd} {\rm Consider a function $f:X\to\oR$ and a pair $(\bar{x},\ox^*)\in\gph\partial f$.  The {\sc combined second-order subdifferential} of $f$ at $\ox$ relative to $\ox^*$ is the set-valued mapping $\breve\partial^2f(\ox,\ox^*):X^{**}\tto X^*$ defined by
\begin{eqnarray}\label{2.11}
\breve\partial^2f(\ox,\ox^*)(u):=\big(\Hat D^*\partial f\big)(\ox,\ox^*)(u)\quad\mbox{ for all }\;u\in X^{**}.
\end{eqnarray}}
\end{Definition}

Our next theorem characterizes tilt stability via the combined second-order subdifferential \eqref{2.11} in Hilbert spaces. Another characterization of this type has been recently established in \cite[Theorem~4.3]{MN} (see Corollary~\ref{coro3} below) by using essentially more complicated arguments.

\begin{Theorem} {\bf(second-order subdifferential characterization of tilt-stable minimizers in Hilbert spaces).}\label{thm3} Let $X$ be a Hilbert space, and let $f\colon X\to\oR$ be both prox-regular and subdifferentially continuous at $\ox$ for $\ox^*=0\in\partial f(\ox)$. Then the following are equivalent:

{\bf (i)} The point $\ox$ is a tilt-stable local minimizer of $f$ with modulus $\kk>0$.

{\bf (ii)} There is $\eta>0$ such that for any $r\in[0,\kk^{-1})$ we have
\begin{eqnarray}\label{4.1}
\begin{array}{ll}
\qquad\qquad\qquad\kk\|u^*\|\ge\|u\|\quad\mbox{and}\quad\la u^*,u\ra\ge-r\|u\|^2\\
\;\mbox{whenever }\;u^*\in\breve\partial^2f(x,x^*)(u)\;\mbox{ with }\;(x,x^*)\in\gph\partial f\cap\B_\eta(\ox,0).
 \end{array}
\end{eqnarray}

{\bf (iii)} There are two constants $\eta>0$ and $r\in[0,\kk^{-1})$ such that \eqref{4.1} holds.
\end{Theorem}
{\bf Proof.} Suppose that {\bf(i)} holds. Employing Theorem~\ref{thm2} and Proposition~\ref{pro1} gives us neighborhoods $V$ of $\ox$ and $V^*$ of $\ox^*$ such that \eqref{3.12} is satisfied. Select $\eta>0$ such that $\B_\eta(\ox)\subset V$. It easily follows from \eqref{3.12} that for all $(x,x^*),\;(v,v^*)\in\gph\partial f\cap(V\times V^*)$ we have the estimate
\begin{eqnarray}\label{4.2}
\|v^*-x^*\|\cdot\|v-x\|\ge\la v^*-x^*,v-x\ra\ge\frac{1}{\kk}\|v-x\|^2,
\end{eqnarray}
which shows that the set $(\partial f)^{-1}(v^*)\cap V$ is a singleton for any $v^*\in V^*$. Pick now $u^*\in\breve\partial^2 f(x,x^*)(u)$ with $(x,x^*)\in\gph\partial f\cap(V\times V^*)$ and get from \eqref{2.11} that for any $\ve>0$ there is some $\delta>0$ with $\B_\delta(x,x^*)\subset V\times V^*$ and such that
\begin{eqnarray}\label{4.3}
\la u^*,v-x\ra-\la u,v^*-x^*\ra\le\ve\big(\|v-x\|+\|v^*-x^*\|\big)\;\mbox{whenever}\;(v,v^*)\in\gph\partial f\cap\B_\delta(x,x^*).
\end{eqnarray}
For small $t>0$ we define $v_t:=(\partial f)^{-1}(v^*_t)\cap V$ with $v^*_t:=x^*+t(u^*-2\kk^{-1}u)\in V^*$ and observe from \eqref{4.2} that $(v_t,v^*_t)\to(x,x^*)$ as $t\dn 0$. Consequently we get that $(v_t,v^*_t)\in\B_\delta(x,x^*)$ for such $t$. Plugging $(v_t,v^*_t)$ into \eqref{4.3} and using \eqref{4.2} ensure then that
\[\begin{array}{ll}
\ve(\kk+1)t\|u^*-2\kk^{-1}u\|&\disp=\ve(\kk+1)\|v_t^*-x^*\|\ge\ve\big(\|v_t-x\|+\|v_t^*-x^*\|\big)\\
&\ge\la u^*, v_t-x\ra-\la u,v_t^*-x^*\ra\\
&\disp=\la t^{-1}(v_t^*-x^*)+2\kk^{-1}u,v_t-x\ra-t\la u,u^*-2\kk^{-1}u\ra\\
&\disp\ge(\kk t)^{-1}\|v_t-x\|^2+2\kk^{-1}\la u,v_t-x\ra-t\la u,u^*-2\kk^{-1}u\ra\\
&\disp\ge(\kk t)^{-1}\|v_t-x\|^2-2\kk^{-1}\|u\|\cdot \|v_t-x\|-t\la u, u^*-2\kk^{-1}u\ra\\
&\disp\ge -t\kk^{-1}\|u\|^2-t\la u,u^*-2\kk^{-1}u\ra=-t\la u^*,u\ra+t\kk^{-1}\|u\|^2,
\end{array}
\]
which yields $\la u^*,u\ra+\ve(\kk+1)\|u^*-2\kk^{-1}u\|\ge\kk^{-1}\|u\|^2$. Letting $\ve\dn 0$ shows that $\la u^*,u\ra\ge\kk^{-1}\|u\|^2$ for all $u^*\in\breve\partial^2 f(x,x^*)(u)$, and then \eqref{4.1} holds whenever $r\in[0,\kk^{-1})$. Thus we arrive at {\bf (ii)}.

Since implication [{\bf(ii)$\Longrightarrow$(iii)}] is obvious, it remains to show that {\bf (iii)} implies {\bf(i)}. Assuming {\bf (iii)} with some constants $r\in[0,\kk^{-1})$ and $\eta>0$ and then using the prox-regularity and subdifferential continuity of $f$ at $\ox$ for $\ox^*$ give us real numbers $\ve,r_\ve>0$ such that
\begin{eqnarray}\label{4.4}
f(x)\ge f(u)+\la u^*,x-u\ra-\frac{r_\ve}{2}\|x-u\|^2\;\mbox{ for all }\;x\in\B_\ve(\ox),(u,u^*)\in\gph\partial f\cap \B_\ve(\ox,\ox^*).
\end{eqnarray}
Fix any $\vartheta>\max\{r_\ve,\kk^{-1}\}>r$ and define $g(x):=f(x)+\frac{\vartheta}{2}\|x-\ox\|^2$ for $x\in X$. Note from the subdifferential sum rule that $\partial g(x)=\partial f(x)+\vt(I-\ox)$, where $I:X\to X$ is the identity mapping. It follows from the open mapping theorem that there is some $\delta>0$ such that $\B_\delta(\ox,0)\subset J(\B_\gg(\ox,0))$ with $\gg:=\min\{\ve,\eta\}$ and $J(x,y):=(x, y+\vt(x-\ox))$ for $(x,y)\in X\times X$. We get from \eqref{4.4} that
\begin{eqnarray}\label{4.5}
g(x)\ge g(u)+\la u^*,x-u\ra+\frac{\vt-r_\ve}{2}\|x-u\|^2\;\mbox{if}\;x\in\B_\delta(\ox),(u,u^*)\in\gph\partial g\cap\B_\delta(\ox,\ox^*).
\end{eqnarray}
Furthermore, the coderivative sum rule from \cite[Theorem~1.62]{M1} implies that for any $u^*\in\breve\partial^2 g(x,x^*)(u)$ with
$(x,x^*)\in\gph\partial g\cap\B_\delta(\ox,0)$ we have
\begin{eqnarray}\label{4.6}
u^*\in\breve\partial^2f\big(x,x^*-\vt(x-\ox)\big)(u)+\vt u.
\end{eqnarray}
Since $J^{-1}(\B_\delta(\ox,0))\subset\B_\gg(\ox,0)$, it ensures that $(x,x^*-\vt(x-\ox))\in\B_\eta(\ox,0)$. This together with \eqref{4.1} and \eqref{4.6} readily implies the estimate
\[
\|u^*\|\cdot\|u\|\ge\la u^*,u\ra\ge(\vt-r)\|u\|^2,
\]
which yields in turn that $\|u^*\|\ge(\vt-r)\|u\|$ for all $u^*\in\breve\partial^2 g(x,x^*)(u)$. Furthermore, it is easy to deduce from \eqref{4.5}  that the set $\gph\partial g$ is locally closed around $(\ox,\ox^*)$ in the norm$\times$norm topology. Therefore the neighborhood coderivative criterion from \cite[Theorem~4.5]{M1} ensures that the subgradient mapping $\partial g$ is metrically regular around $(\ox,0)$ with modulus $(\vt-r_1)^{-1}$ for some $r_1\in(r,\kk^{-1})$. Combining this with \eqref{4.5} and applying Corollary~\ref{coro1} to the function $g$, we find neighborhoods $U$ of $\ox$ and $U^*$ of $\ox^*$ such that $U\times U^*\subset\B_\delta(\ox,0)$ and that
\begin{eqnarray}\label{4.7}
g(x)\ge g(u)+\la u^*,x-u\ra+\frac{\vt-r_1}{2}\|x-u\|^2\;\mbox{for all}\;x\in U,\,(u,u^*)\in\gph\partial g\cap(U\times U^*).
\end{eqnarray}
Consider $W:=J^{-1}(U\times U^*)$, a neighborhood of $(\ox,0)$, and get from \eqref{4.7} similarly to \eqref{3.16} that
\begin{eqnarray}\label{4.8}
f(x)\ge f(u)+\la u^*,x-u\ra-\frac{r_1}{2}\|x-u\|^2\;\mbox{for all}\;x\in U,\,(u,u^*)\in\gph\partial f\cap W
\end{eqnarray}
with $r_1<\kk^{-1}$. It follows from \eqref{4.1} and \cite[Lemma~6.4]{MN} that $\partial f$ is metrically regular around $(\ox,0)$ with modulus $\kk$. Employing finally Corollary~\ref{coro1} and Proposition~\ref{pro1} allows us to conclude that $\ox$ is a tilt-stable minimizer of $f$ with modulus $\kk$.\endproof

Observe that the first inequality in \eqref{4.1} means that the subgradient mapping $\partial f$ is metrically regular around $(\ox,0)$ due to the neighborhood coderivative criterion of \cite[Theorem~4.5]{M1}; this part is similar to \eqref{4.9} from \cite[Theorem~4.3]{MN}. On the other hand, the second estimate in \eqref{4.1} is new and it is indeed the key for establishing the full set of equivalences in our major Theorem~\ref{coro5} below. Note further that the restriction on $r\in [0,\kk^{-1})$ is {\em essential} in Theorem~\ref{thm3} and its consequences as well as in the previous results of this paper. Indeed, it is easy to verify that for any given modulus $\kk>0$ the function $f(x):=-\frac{1}{2\kk}\|x\|^2$ on an arbitrary Hilbert space $X$ satisfies condition \eqref{4.1} with $\ox=0\in X$ and every $r\ge\kk^{-1}$, but the point $\ox$ is {\em not even a local minimizer} of $f$.\vspace*{0.05in}

Since the result of \cite[Theorem~4.3]{MN} gives yet another characterization of tilt-stable minimizers in the setting of Theorem~\ref{thm3}, it must be equivalent to both conditions in (ii) and (iii) of the latter theorem. However, we cannot see a direct proof of the major part in this equivalence and recover now the result of \cite[Theorem~4.3]{MN} by using the arguments of Theorem~\ref{thm3}, which are much simpler than the original proof in \cite{MN}.

\begin{Corollary}{\bf (another second-order characterization of tilt stability in Hilbert spaces).}\label{coro3} In the setting of Theorem~{\rm\ref{thm3}} assertions {\bf (i)}--{\bf(iii)} are equivalent to:

{\bf (iv)} There is a real number $\eta>0$ such that
\begin{eqnarray}\label{4.9}
\la u^*,u\ra\ge\frac{1}{\kk}\|u\|^2\;\mbox{whenever }\;u^*\in\breve\partial^2 f(x,x^*)(u)\;\mbox{ with }\;(x,x^*)\in\gph\partial f\cap\B_\eta(\ox,0).
\end{eqnarray}
\end{Corollary}
{\bf Proof.} The verification of the major implication [{\bf(i)}$\Longrightarrow${\bf(iv)}] follows the same lines as in the proof of [{\bf(i)}$\Longrightarrow${\bf(ii)}] in Theorem~\ref{thm3}. It remains to show that {\bf (iv)} yields assertion {\bf(ii)} of this theorem. Indeed, for any $u^*\in\breve\partial^2 f(x,x^*)(u)$ with $(x,x^*)\in\gph\partial f\cap\B_\eta(\ox,0)$ we get from \eqref{4.9} that
\[
\kk\|u^*\|\cdot\|u\|\ge\kk\la u^*,u\ra\ge\kk\|u\|^2,
\]
which implies that $\kk\|u^*\|\ge \|u\|$. This together with \eqref{4.9} verifies \eqref{4.1}.\endproof

When $X$ is finite-dimensional, we can derive from Theorem~\ref{thm3} new {\em pointwise quantitative} characterizations of tilt-stable minimizers via the generalized Hessian \eqref{2.10} at the reference point.

\begin{Corollary} {\bf(pointwise quantitative characterization of tilt stability in finite dimensions).}\label{coro4} Let $X=\R^n$ in the setting of Theorem~{\rm\ref{thm3}} and consider the following two statements:

{\bf (i)} The point $\ox$ is a tilt-stable local minimizer of the function $f$ with modulus $\kk$.

{\bf(ii)} There are numbers  $\mu>0$ and $r\in[0,\mu^{-1})$ such that for any $u\in X$ we have
\begin{eqnarray}\label{4.10}
\mu\|u^*\|\ge\|u\|\quad\mbox{and}\quad\la u^*,u\ra\ge-r\|u\|^2\quad\mbox{whenever}\quad u^*\in\partial^2f(\ox,0)(u).
\end{eqnarray}
Then implication $[${\bf (i)}$\Longrightarrow${\bf (ii)}$]$ holds with $\mu=\kk$ and any $r\in[0,\kk^{-1})$ while the converse implication is satisfied with any modulus $\kk>\mu$.
\end{Corollary}
{\bf Proof.} Assume that {\bf(i)} holds and deduce from Theorem~\ref{thm3}(ii) the validity of \eqref{4.1} with some $\eta>0$ for any $r\in[0,\kk^{-1})$. Pick any $(u,u^*)\in\gph\partial^2 f(\ox,0)$ and, arguing by contradiction, suppose that it does not satisfy at least one of the inequalities in \eqref{4.10}. Then the limiting relationship between the second-order constructions \eqref{2.10} and \eqref{2.11} induced by that in \eqref{2.3a} for the normal cones \eqref{2.3} allows us to approximate $(u,u^*)$ in $\R^n\times\R^n$ by some sequence of $(u_k,u^*_k)\in\gph\breve\partial^2f(x_k,x^*_k)$ with $(x_k,x^*_k)\in\gph\partial f\cap\B_\eta(\ox,0)$ such that the corresponding inequality in \eqref{4.1} does not hold for $(u_k,u^*_k)$ as $k\to\infty$. This contradiction justifies {\bf (ii)} with $\mu=\kappa$ and any $r\in[0,\kk^{-1})$.

Now suppose that {\bf(ii)} is satisfied. Choosing any $\mu_0\in(\mu, r^{-1})$, we get  from {\bf(ii)} by the (pointwise) coderivative criterion from  \cite[Corollary~4.3]{mor93} and \cite[Theorem~9.43]{RW} that the subgradient mapping $\partial f$ is metrically regular around $(\ox,0)$ with modulus $\mu_0$. Let us proceed as in the proof of Theorem~\ref{thm3} and define $g(x):=f(x)+\frac{\vt}{2}\|x-\ox\|^2$ with some $\vt>\max\{\mu_0^{-1},r_\ve\}>r$ and $\ve, r_\ve$ taken from \eqref{4.4}. Then we deduce that $\partial g(x)=\partial f(x)+\vt(I-\ox)$. It follows from \cite[Theorem~1.62]{M1} that $u^*-\vt u\in\partial^2 f(\ox,0)(u)$ for any $u^*\in\partial^2g(\ox,0)(u)$. This together with \eqref{4.10} gives us that
\[
\|u^*\|\cdot\|u\|\ge\la u^*,u\ra\ge(\vt-r)\|u\|^2,
\]
which yields in turn that $\|u^*\|\ge(\vt-r)\|u\|$. Employing again the aforementioned coderivative criterion ensures that the mapping $\partial g$ is metrically regular around $(\ox,0)$ with modulus $(\vt-r_1)^{-1}$ for some $r_1\in(r,\mu_0^{-1})$. Similarly to \eqref{4.8} we find neighborhoods $U$ of $\ox$ and $W$ of $(\ox,\ox^*)$ with
\begin{eqnarray*}
f(x)\ge f(u)+\la u^*,x-u\ra-\frac{r_1}{2}\|x-u\|^2\quad\mbox{for all}\quad x\in U,\,(u,u^*)\in\gph\partial f\cap W.
\end{eqnarray*}
Since $r_1\mu_0<1$, combining Corollary~\ref{coro1} and Proposition~\ref{pro1} gives us that $\ox$ is a tilt-stable local minimizer of $f$ with modulus $\mu_0$. Taking into account that $\mu_0$ was chosen arbitrarily in $(\mu,r^{-1})$, we conclude that $\kk$ in {\bf(i)} can be chosen arbitrary in $(\mu,\infty)$, which completes the proof. \endproof

Finally, we are ready to derive the culminating result of this section that establishes appropriate nonsmooth counterparts of the equivalent conditions {\bf (a)}--{\bf (d)} formulated above for ${\cal C}^2$-smooth functions. It confirms, in particular, that the generalized Hessian \eqref{2.10} can appropriately substitute the classical Hessian for the large class of prox-regular and subdifferentially continuous functions on $\R^n$, which plays a pivoting role in second-order variational analysis.

\begin{Theorem}{\bf(equivalent second-order conditions for prox-regular functions in finite dimensions).}\label{coro5} Let $f\colon\R^n\to\overline{\R}$ be both prox-regular and subdifferentially continuous at $\ox$ for $0\in\partial f(\ox)$. Then the following conditions are equivalent:

{\bf (a)} The subgradient mapping $\partial f\colon\R^n\tto\R^n$ is metrically regular around $(\bar{x},0)$ and the generalized Hessian $\partial^2 f(\ox,0)$ is positive-semidefinite in the sense that
\begin{eqnarray}\label{4.11}
\la u^*,u\ra\ge 0\quad\mbox{whenever}\quad u^*\in\partial^2f(\ox,0)(u),\;u\ne 0.
\end{eqnarray}

{\bf (b)} The subgradient mapping $\partial f$ is strongly metrically regular around $(\bar{x},0)$ and $\ox$ is a local minimizer of $f$.

{\bf (c)} The generalized Hessian $\partial^2f(\ox,0)$ is positive-definite in the sense of \eqref{4.11a}.

{\bf (d)} ${\rm Ker}\,\partial^2f(\ox,0)=\{0\}$ and $\partial^2f(\ox,0)$ is positive-semidefinite in the sense of \eqref{4.11}.\\[1ex]
Furthermore, each of the equivalent conditions {\bf (a)}--{\bf (d)} is necessary and sufficient for the point $\ox$ being a tilt-stable local minimizer of the function $f$.
\end{Theorem}
{\bf Proof.} As mentioned above, the subgradient mapping $\partial f$ for each $f$  from the class of functions under consideration is locally closed around $(\ox,0)$, and thus the equivalence between conditions {\bf (a)} and {\bf (d)} follows directly from the coderivative criterion without any additional assumption; see the detailed discussion after formula \eqref{ker}. By Corollary~\ref{coro4} condition {\bf (d)} is equivalent to $\ox$ being a tilt-stable local minimizer of $f$. This ensures that {\bf (d)} is equivalent to {\bf (b)} by Proposition~\ref{pro1}. By [{\bf(i)}$\Longrightarrow${\bf(ii)}] in Corollary~\ref{coro4} with $r=0$, we deduce that tilt stability of $\ox$ yields the validity of {\bf (c)}. Since implication [{\bf (c)}$\Longrightarrow${\bf (d)}] is obvious, this completes the proof of the theorem.\endproof

If $\ox$ is a local minimizer of $f$, then $0\in\partial f(\ox)$ and we have all the equivalences of Theorem~\ref{coro5} for prox-regular and subdifferentially continuous functions. If furthermore $f$ is ${\cal C}^2$-smooth around $\ox$, then the generalized Hessian $\partial^2f(\ox,0)$ reduces to the classical one \eqref{c2} and its positive-semidefiniteness \eqref{4.11} holds automatically at local minimizers. Hence in this case the equivalent conditions in Theorem~\ref{coro5} go back to those {\bf (a)}--{\bf(d)} listed at the beginning of this section.

We now observe in passing that the positive semi-definiteness condition \eqref{4.11} is not a necessary condition for optimality at local minimizers of prox-regular and subdifferentially regular functions. In fact it is not necessary even within a very special subclass of it, bridging between smoothness and convexity while covering at the same time a great many of functions that are of interest as the essential objective in minimization problems. Recall \cite[Definition~10.23]{RW} that $f\colon\R^n\to\oR$ is {\em fully amenable} at $\ox\in\dom f$  if there exist a mapping $h\colon\R^n\to\R^m$ that is ${\cal C}^2$-smooth around $\ox$ and a piecewise linear-quadratic function $\ph\colon\R^m\to\oR$ such that $f=\ph\circ h$ around $\ox$ and
\begin{equation}\label{fa}
\partial^\infty\ph\big(h(\ox)\big)\cap{\rm Ker}\,\nabla h(\ox)^*=\{0\},
\end{equation}
where the {\em singular/horizon subdifferential} of $\ph\colon\R^m\to\oR$ at $\oy\in\dom\ph$ is defined by
$$
\partial^\infty\ph(\oy):=\big\{v\in\R^m\big|\;(v,0)\in N_{\epi\ph}\big(\oy,\ph(\oy)\big)\big\}.
$$
We refer the reader to \cite{M1,RW} for the latter subdifferential construction and to \cite{RW} for properties and applications of fully amenable functions.

The following illustrative example shows that the semidefiniteness property \eqref{4.11} of the generalized Hessian is {\em not a necessary condition} for local minimizers of fully amenable functions $f$.

\begin{Example}{\bf (positive-semidefiniteness of the generalized Hessian is not necessary for optimality)}.\label{ex}{\rm Define the function $f:\R^2\to\oR$ by
\begin{equation}\label{ex-f}
f(x):=(x^2_1-x^2_2)+\delta_\O(x_1,x_2)\;\mbox{ with }\;\O:=\big\{(x_1,x_2)\in\R^2\big|\;-x_1+x_2\le0, -x_1-x_2\le0\big\}.
\end{equation}
Observe that $\ox:=(0,0)$ is a {\em local minimizer} of $f$. Moreover, it follows from the definition that $f$ is {\em fully amenable} at $\bar{x}$ since it is the sum of a ${\cal C}^2$-smooth function and the fully amenable composition $\delta_\O=\delta_{\R^2_+}\circ h$ with  $h(x_1,x_2):=(x_1-x_2,x_1+x_2)$. By the sum rule for second-order subdifferentials \cite[Proposition~1.121]{M1} we get
\begin{eqnarray}\label{nn}
\partial^2 f(\ox,0)(u)=(2u_1,-2u_2)+\partial^2\delta_\O(\ox,0)(u)\quad \mbox{for all}\quad u\in \R^2.
\end{eqnarray}
Note further that $\O$ is a polyhedral in $\R^2$ and the {\em critical cone} $K_\O(\ox,0):=N_\O(\ox)^*\cap\{0\}^\perp$ is $\O$. Due to \cite[Proposition~4.4]{PR2} we have $0\in\partial^2\delta_\O(\ox,0)(\ou)$ with $\ou:=(0,1)$, since $\ou\in K_\O(\ox,0)-K_\O(\ox,0)=\R^2$ and $0\in (K_\O(\ox,0)-K_\O(\ox,0))^*$. It follows from \eqref{nn} that $(0,-2)\in \partial^2 f(\ox,0)(\ou)$.

Since $\la (0,-2),\ou\ra=-2<0$,  the positive-semidefiniteness condition \eqref{4.11} is {\em violated} at $\ou$. Thus it is {\em not} a necessary condition for local optimality in problems of minimizing nonsmooth fully amenable functions on $\R^2$.}
\end{Example}

\section{Concluding Remarks}
\setcounter{equation}{0}

In this paper we established close relationships between some versions of second-order/quadratic growth for extended-real-valued functions and fundamental metric regularity and subregularity properties of their subdifferentials. Then we used these relationships to develop new second-order characterizations of tilt stability of local minimizers for a general class of prox-regular and subdifferentially continuous functions, which have been well recognized as an appropriate polygon of second-order variational analysis in the absence of the classical ${\cal C}^2$-smoothness. In this way we extended several equivalent second-order conditions from the ${\cal C}^2$-smooth case to the prox-regular one under an additional assumption formulated via the generalized Hessian.
We moreover conjectured that this assumption was necessary. We aim to investigate this conjecture in future work.

There are two natural lines of future research in the direction of this paper. One is to implement and further develop the results obtained here for minimizing general extended-real-valued functions to specific models of constrained optimization, e.g., nonlinear and extended nonlinear programming, conic programming, mathematical programs with equilibrium constraints, etc. Some results of this type related, in particular, to previous characterizations of tilt stability have been recently obtained in \cite{MN,MN1,MO,MOR,MR}. Another appealing line of research is to extend the obtained new characterizations of tilt stability in optimization to the more general and important case of {\em full stability}, the notion introduced by Levy, Poliquin and Rockafellar in \cite{LPR}; see also \cite{mrs} for recent results on full stability in the framework of extended-real-valued functions and their applications to particular classes of problems in finite-dimensional constrained optimization.
\bigskip

{\bf Acknowledgements.} The authors thank Steve Robinson for pointing out an inaccuracy in an example appearing in an earlier version of the manuscript.

\end{document}